\documentclass{amsart}

\usepackage{amssymb}
\usepackage[utf8]{inputenc}

\usepackage{eucal} 
\usepackage{stmaryrd} 

\usepackage{enumerate}
\usepackage[numbers,longnamesfirst]{natbib}
\usepackage[bookmarks,colorlinks=true,citecolor=blue]{hyperref}

\theoremstyle{plain}
\newtheorem{theorem}{Theorem}[section]
\newtheorem{lemma}[theorem]{Lemma}
\newtheorem*{claim*}{Claim}
\newtheorem{corollary}[theorem]{Corollary}
\newtheorem{proposition}[theorem]{Proposition}

\theoremstyle{definition}
\newtheorem{definition}[theorem]{Definition}

\theoremstyle{remark}

\numberwithin{equation}{section}

\title{Measures and Their Random Reals}

\author{Jan Reimann}

\address{Department of Mathematics  \\
  Pennsylvania State University,  
  University Park, PA, USA\\
}
\thanks{Reimann was partially supported by NSF grants DMS-0801270 and DMS-1201263.}

\email{reimann@math.psu.edu}

\author{Theodore A. Slaman} 

\address{Department of Mathematics  \\
  University of California at Berkeley \\
  Berkeley, CA, USA
}
\thanks{Slaman was partially supported by NSF grants DMS-0501167 and DMS-1001551.}

\email{slaman@math.berkeley.edu}

\date{}

\newcommand{\Bit}{\ensuremath{\{0,1\}}}
\newcommand{\Nat}{\ensuremath{\mathbb{N}}}

\newcommand{\Rat}{\ensuremath{\mathbb{Q}}}
\newcommand{\Real}{\ensuremath{\mathbb{R}}}
\newcommand{\Cant}{\ensuremath{2^{\omega}}}
\newcommand{\Baire}{\ensuremath{\omega^{\omega}}}
\newcommand{\Str}[1][<\omega]{\ensuremath{2^{#1}}}

\newcommand{\Sle}{\ensuremath{\sqsubset}}
\newcommand{\Sleq}{\ensuremath{\sqsubseteq}}

\newcommand{\Sgeq}{\ensuremath{\sqsupseteq}}

\newcommand{\Cl}[1]{\ensuremath{#1}}

\newcommand{\Cyl}[1]{\ensuremath{\llbracket #1 \rrbracket}}
\newcommand{\ACyl}[1]{\ensuremath{\llbracket #1 \rrbracket}}

\newcommand{\Fam}[1]{\ensuremath{\mathcal{#1}}}
\newcommand{\Meas}{\ensuremath{\mathcal{P}(\Cant)}}

\newcommand{\Rest}[1]{\ensuremath{\restriction {#1}}}
\newcommand{\Op}[1]{\ensuremath{\operatorname{#1}}}
\newcommand{\Halt}{\ensuremath{\emptyset'}}
\newcommand{\Conc}{\ensuremath{\mbox{}^\frown }}
\newcommand{\Estr}{\ensuremath{\epsilon}}
\newcommand{\Tup}[1]{\ensuremath{\langle #1 \rangle}}

\newcommand{\Leb}{\ensuremath{\lambda}}

\DeclareMathOperator{\NCR}{NCR}
\DeclareMathOperator{\T}{T}
\DeclareMathOperator{\TT}{tt}
\DeclareMathOperator{\WTT}{wtt}
\DeclareMathOperator{\meas}{meas}

\begin{document}

\begin{abstract}
	We study the randomness properties of reals with respect to arbitrary probability measures on Cantor space. We show that every non-computable real is non-trivially random with respect to some measure. The probability measures constructed in the proof may have atoms. If one rules out the existence of atoms, i.e.\ considers only continuous measures, it turns out that every non-hyperarithmetical real is random for a continuous measure. On the other hand, examples of reals not random for any continuous measure can be found throughout the hyperarithmetical Turing degrees.
\end{abstract}

	\maketitle

%
%
\section{Introduction} \label{sec:intro}

Over the past decade, the study of algorithmic randomness has produced an impressive number of results. The theory of Martin-Löf random reals, with all its ramifications (e.g.\ computable or Schnorr randomness, lowness and triviality) has found deep and significant applications in computability theory, many of which are covered in recent books by \citet{downey-hirschfeldt:ip} and \citet{nies:ta}. Usually, the measure for which randomness is considered in these studies is the uniform $(1/2,1/2)$-measure on Cantor space, which is measure theoretically isomorphic to Lebesgue measure on the unit interval.

However, one may ask what happens if one changes the underlying measure. It is easy to define a generalization of Martin-L\"of tests which allows for a definition of randomness with respect to arbitrary \emph{computable} measures. 
For arbitrary measures, the situation is more complicated. 
\citet{martinloef:1966} studied randomness for arbitrary Bernoulli
measures, \citet{levin:1973, levin:1976, levin:1984} studied arbitrary
measures on $\Cant$, while \citet{gacs:2005} generalized Levin's
approach to a large class of computable metric spaces. Most recently, \citet{hoyrup-rojas:computabilityprobability_2009} showed that Levin's approach can be extended to any computable metric space. It also turned out that much of the theory (i.e.\ existence of a universal test, the connection with descriptive complexity etc.) can be preserved under reasonable assumptions for the underlying space.

\medskip
In this paper, we study the following question: Given a real $x \in \Cant$, with respect to which measures is $x$ random? This can be seen as a dual to the usual investigations in algorithmic randomness: Given a (computable) measure $\mu$ (where $\mu$ usually is the uniform distribution), what are the properties of a $\mu$-random real?

Of course, every real $x$ is trivially random with respect to a measure $\mu$  which assigns some positive mass to it (as a singleton set), i.e.\ $\mu \{x\} > 0$. But one can ask if there is a measure for which this is not the case and $x$ is still random. It turns out that this is possible precisely for the non-computable reals. Furthermore, one could ask whether there exists a measure $\mu$ such that $x$ is $\mu$-random and $\mu$ does not have atoms at all, i.e. $\mu \{y\} = 0$ for all $y$. The answer to this question shows an unexpected correspondence between the randomness properties of a real and its complexity in terms of effective descriptive set theory: If $x$ is not $\Delta^1_1$, then there exists a non-atomic measure with respect to which $x$ is random.

These results motivate a further investigation. What is the exact classification of all reals (inside $\Delta^1_1$) which are not random with respect to any continuous measure? If we look at $n$-randomness ($n \geq 2$), what is the size and structure of the reals that are not $n$-random for some continuous measure? The latter question will be studied in separate paper \citep{reimann-slaman:ip}.

%
%
\section{Transformations and measures on Cantor Space} \label{sec:meas-cant}

In this section we first quickly review the basic notions of Turing functionals and of measures on Cantor space $\Cant$. The space of probability measures on $\Cant$ is compact Polish, and we can devise a suitable effective representation of measures in terms of Cauchy sequences with respect to a certain metric. This enables us to code measures as binary sequences, so we can use them as oracles in Turing machine computations. This way we can extend Martin-L\"of's notion of randomness to arbitrary measures by requiring that a Martin-L\"of test for a measure is uniformly enumerable in a representation of the measure. This will be done in Section \ref{sec:rand-trans}.

%
%
\subsection{The Cantor space as a metric space}

The \emph{Cantor space} $\Cant$ is the set of all infinite binary sequences, also called \emph{reals}. The usual metric on $\Cant$ is defined as follows: Given $x,y \in \Cant$, $x \neq y$, let $x \sqcap y$ be the longest common initial segment of $x$ and $y$ (possibly the empty string $\Estr$). Define
\[
	d(x,y) = \begin{cases}
		2^{-|x \sqcap y|} & \text{if $x \neq y$}, \\
		0 & \text{if $x=y$}.
	\end{cases}
\]
Endowed with this metric, $\Cant$ is a compact Polish space. A countable basis is given by the \emph{cylinder sets}
\[
	\Cyl{\sigma} = \{ x : \: x\Rest{|\sigma|} \: = \: \sigma\},
\]
where $\sigma$ is a finite binary sequence (string), and $|\sigma|$ denotes the length of $\sigma$. We use $\Str$ to denote the set of all finite binary sequences, and we use $\Sleq$ to denote the usual prefix partial ordering between finite strings. This partial ordering extends in a natural way to $\Str \cup \Cant$. Thus, $x \in \Cyl{\sigma}$ if and only if $\sigma \Sle x$. Finally, given $U \subseteq \Str$, we write $\Cyl{U}$ to denote the open set induced by $U$, i.e. $\Cyl{U} = \bigcup_{\sigma \in U} \Cyl{\sigma}$.

\subsection{Turing functionals} 
\label{sub:turing_functionals}

The notion of a Turing functional will be important in this paper, so we give a formal definition and explain how functionals give rise to partial, continuous mappings from $\Cant$ to $\Cant$.

A \emph{Turing functional} $\Phi$ is a computably enumerable set of triples $(m,k,\sigma)$
such that $m$ is a natural number, $k$ is either $0$ or $1$, and $\sigma$ is a
finite binary sequence. Further, for all $m$, for all $k_{1}$ and $k_{2}$, and
for all compatible $\sigma_{1}$ and $\sigma_{2}$, if $(m,k_{1},\sigma_{1}%
)\in\Phi$ and $(m,k_{2},\sigma_{2})\in\Phi$, then $k_{1}=k_{2}$ and
$\sigma_{1}=\sigma_{2}$. 

In the following, we will also assume that Turing functionals $\Phi$ are \emph{use-monotone}, which means the following hold.
\begin{enumerate}
\item  For all $(m_{1},k_{1},\sigma_{1})$ and $(m_{2},k_{2},\sigma_{2})$ in
$\Phi$, if $\sigma_{1}$ is a proper initial segment of $\sigma_{2}$, then
$m_{1}$ is less than $m_{2}$.

\item  For all $m_{1}$ and $m_{2}$, $k_{2}$ and $\sigma_{2}$, if $m_{2}>m_{1}$
and $(m_{2},k_{2},\sigma_{2})\in\Phi$, then there are $k_{1}$ and $\sigma_{1}$
such that $\sigma_{1}\Sleq\sigma_{2}$ and $(m_{1},k_{1},\sigma_{1})\in
\Phi$.
\end{enumerate}

We write $\Phi^\sigma(m)=k$ to indicate that there is a $\tau$ such
that $\tau$ is an initial segment of $\sigma$, possibly equal to $\sigma$, and
$(m,k,\tau)\in\Phi$. In this case, we also write $\Phi^\sigma(m)\downarrow$, as opposed to $\Phi^\sigma(m) \uparrow$, indicating that for all $k$ and all $\tau \Sleq \sigma$, $(m,k,\tau) \not\in \Phi$.

If $x\in\Cant$, we write $\Phi^x(m)=k$ to indicate
that there is an $l$ such that $\Phi^{x\Rest{l}}(m)=k$. This way, for given $x \in \Cant$, $\Phi^x$ defines a partial function from $\omega$ to $\{0,1\}$ (identifying reals with sets of natural numbers). If this function is total, it defines a real $y$, and in this case we write $\Phi(x) = y$ and say that $y$ is Turing reducible to $x$ via $\Phi$, $y \leq_{\T} x$. 

By use-monotonicity, if $\Phi^\sigma(m) \downarrow$, then $\Phi^\sigma(n)\downarrow$ for all $n < m$. If we let $\overline{m}$ be maximal such that $\Phi^\sigma(\overline{m}) \downarrow$, $\Phi^\sigma$ gives rise to a string $\tau$ of length $\overline{m}+1$,
\[
 	\tau = \Phi^\sigma(0) \dots \Phi^\sigma(\overline{m}).
\] 
If $\Phi^\sigma(n)\uparrow$ for all $n$, we put $\tau = \Estr$. We write $\Phi(\sigma) = \tau$. This way a Turing functional induces a function from $\Str$ to $\Str$ that is \emph{monotone}, that is, $\sigma \Sleq \tau$ implies $\Phi(\sigma) \Sleq \Phi(\tau)$. Note that $\Phi(\sigma)$ is not necessarily a computable function, but we can effectively approximate it by prefixes. More precisely, there exists a computable mapping $(\sigma,s) \mapsto \Phi_s(\sigma) \in \Str \cup \Cant$ so that
$\Phi_s(\sigma) \Sleq \Phi_{s+1}(\sigma)$, $\Phi_s(\sigma) \Sleq \Phi_s(\sigma\Conc i)$ ($i \in \{0,1\}$), and $\lim_s \Phi_s(\sigma) = \Phi(\sigma)$. For technical reasons that will become clear in Section \ref{sec:arbi-rand}, we also require $\Phi_s(\sigma)$ to have the following properties.
\begin{enumerate}[(a)]
  \item $|\Phi_s(\sigma)| \leq s$ (the approximation does not grow too quickly in length), and

  \item $|\Phi_{s+1}(\sigma\Conc i)| \leq |\Phi_s(\sigma)| +1$, for any $i \in \{0,1\}$  (one more unit of time and one more bit of information will yield at most one additional bit of output).
\end{enumerate}

If, for a real $x$, $\lim_n |\Phi(x\Rest{n})| = \infty$, then $\Phi(x) = y$, where $y$ is the unique real that extends all $\Phi(x\Rest{n})$. In this way, $\Phi$ also induces a partial, continuous function from $\Cant$ to $\Cant$. We will use the same symbol $\Phi$ for the Turing functional, the monotone function from $\Str$ to $\Str$, and the partial, continuous function from $\Cant$ to $\Cant$. It will be clear from the context which $\Phi$ is meant. 

A Turing functional $\Phi$ has \emph{computably bounded use} if there exists a computable function $g:\Nat \to \Nat$ so that  $(m,k,\sigma) \in \Phi$ implies that $|\sigma| \leq g(m)$. If $\Phi(x) = y$ for such a functional, we say that $y$ is \emph{bounded Turing} or \emph{weak truth-table} reducible to $x$, $y \leq_{\WTT} x$.

Turing functionals can be relativized with respect to a parameter $z$, by requiring that $\Phi$ is c.e.\ in $z$. We call such functionals \emph{Turing $z$-functionals}. This way we can consider relativized Turing reductions. A real $x$ is Turing reducible to a real $y$ relative to a real $z$, written $x \leq_{\T(z)} y$, if there exists a Turing $z$-functional $\Phi$ such that $\Phi(x) = y$.


%
%
\subsection{Probability measures}

A \emph{measure} on $\Cant$ is a \emph{countably additive,
  monotone function} $\mu: \Fam{F} \to [0,1]$, where $\Fam{F}
\subseteq \mathcal{P}(\Cant)$ is a $\sigma$-algebra. If $\mu$ is normalized, i.e.\ if $\mu(\Cant) =
1$, then $\mu$ is called a \emph{probability measure}. A measure $\mu$ is a \emph{Borel measure} if $\Fam{F}$ is
the Borel $\sigma$-algebra on $\Cant$. It is a basic result of measure
theory that a measure with domain $\Fam{F}$ is uniquely determined by the values
it takes on an algebra $\Fam{A} \subseteq \Fam{F}$ that generates
$\Fam{F}$. It is not hard to see that in $\Cant$, the Borel sets are generated by
the algebra of \emph{clopen sets}, i.e.\ finite unions of basic open
cylinders. Normalized, countably additive, monotone set functions on
the algebra of clopen sets are induced by any function $\rho: \Str \to
[0,1]$ satisfying
\begin{equation} \label{equ-probability measure}
  \rho(\Estr) = 1 \text{ and for all $\sigma \in \Str$, } \rho(\sigma) = \rho(\sigma \Conc 0) + \rho(\sigma \Conc 1),
\end{equation}
where $\Estr$ denotes the empty string. If $\rho$ is as in \eqref{equ-probability measure}, then putting $\mu(\Cyl{\sigma}) =
\rho(\sigma)$ induces a monotone, additive function on the clopen
sets, which in turn uniquely extends to a Borel probability measure on
$\Cant$. In the following, we will deal exclusively with Borel measures. Thus, when we speak of
\emph{measures}, we will always mean Borel measures. 
If convenient, we write $\mu A$ for $\mu(A)$ to improve readability.

The \emph{Lebesgue measure} $\Leb$ on $\Cant$ is obtained by
distributing a unit mass uniformly along the paths of $\Cant$, i.e.\
by setting $\Leb(\sigma) = 2^{-|\sigma|}$. A \emph{Dirac measure}, on
the other hand, is defined by putting a unit mass on a single real,
i.e.  for $x \in \Cant$, let
\[
\delta_x\Cyl{\sigma} = \begin{cases}
  1 & \text{if } \sigma \Sle x, \\
  0 & \text{otherwise.}
\end{cases}
\]
If, for a measure $\mu$ and $x \in \Cant$, $\mu\{x\} > 0$, then $x$
is called an \emph{atom} of $\mu$. Obviously, $x$ is an atom of
$\delta_x$. A measure that does not have any atoms is called
\emph{continuous}.

%
%
\subsection{The space of probability measures on Cantor space}

We denote by $\Meas$ the set of all probability measures on $\Cant$. $\Meas$ can be given a topology (the so-called \emph{weak-$*$ topology}) by letting $\mu_n \to \mu$ if $\int f d\mu_n \to \int f d\mu$ for all continuous real-valued functions $f$ on $\Cant$. 

It is known that if $X$ is compact metrizable, then so is the space of all probability measures on $X$ (see for instance \citep{kechris:1995}). Therefore, 
$\Meas$ is compact metrizable. In particular, it is Polish. A compatible metric (see for example \citep{glasner:2003}) is given by
\[
	d_{\meas}(\mu,\nu) = \sum_{n=1}^\infty 2^{-n} d_n(\mu,\nu),
\]
where
\[
	d_n(\mu,\nu) = \frac{1}{2} \sum_{|\sigma| = n} |\mu\Cyl{\sigma} - \nu\Cyl{\sigma}|.
\]

A countable, dense
subset $\mathcal{D} \subseteq \Meas$ is given by the set of measures which assume
positive, dyadic rational values (of the form $m/2^n$ with $m,n \geq 0$) on a finite number of dyadic rationals, i.e. $\mathcal{D}$ is the set
of measures of the form
\begin{equation*}
  \nu_{\Delta, Q} = \sum_{\sigma \in \Delta} Q(\sigma) \delta_{\sigma \Conc 0^\omega},
\end{equation*}
where $\Delta$ is a finite set of finite strings (representing dyadic rational numbers) and $Q: \Delta \to [0,1]$ such that $\sum_{\sigma \in \Delta} Q(\sigma) = 1$ and  for all $\sigma \in \Delta$, $Q(\sigma)$ is a dyadic rational number.  

It is straightforward to show that in case $\mu,\nu$ are restricted to measures in $\mathcal{D}$, the following relations are computable:
\[
	d_{\meas}(\mu,\nu) < q \quad \text{ and } \quad d_{\meas}(\mu,\nu) \leq q \qquad (q \in \Rat \cap[0,1]).
\]
By effectively enumerating all possible combinations $(\Delta,Q)$, we can also effectively enumerate the set $\mathcal{D}$. In the following, we fix such an enumeration $\mathcal{D} = \{\nu_0, \nu_1, \nu_2, \dots \}$. The triple $(\Meas,\mathcal{D},d_{\meas})$ forms a \emph{computable metric space} (see e.g.\ \citep{weihrauch:2000}, or \cite{hoyrup-rojas:computabilityprobability_2009}, \citep{gacs:2005} in this particular context). 

We can represent measures in $\Meas$ through Cauchy sequences of measures in $\mathcal{D}$, which in turn can be encoded as reals using the enumeration of $\mathcal{D}$. Furthermore, the fact that $\Meas$ is a compact, computable metric space can be used to devise a coding scheme that captures the topology of $\Meas$.

\begin{proposition} \label{prop:exist_representation}
 	There exists a Turing functional\footnote{The way we defined Turing functionals in Section \ref{sub:turing_functionals}, they induce partial mappings from $\Cant$ to $\Cant$. Here we obviously assume that $\Gamma$ induces a mapping from $\Cant$ to $\Baire$. The definition given in Section \ref{sub:turing_functionals} is easily adjustable to this case.} $\Gamma$ so that for all $x\in \Cant$ and all $n \in \Nat$, $\Gamma^x(n)$ is defined and 
 	\[
 		d_{\meas}(\nu_{\Gamma^x(n)}, \nu_{\Gamma^x(n+1)}) \leq 2^{-n}.
 	\]
 	Furthermore, $\Gamma$ induces a continuous surjection $\rho: \Cant \to \Meas$ by letting
 	\[
 		\rho(x) = \lim_n \nu_{\Gamma^x(n)},
 	\]
 	where the limit is taken with respect to the weak-$*$ topology. Finally, $\rho$ is such that for any $x \in \Cant$, the set
 	\[
 		\rho^{-1}(\{\rho(x)\})
 	\]
 	is $\Pi^0_1(x)$.
\end{proposition} 
For a proof, see \citep{day-miller:randomness-non-computable_2011}. (\citep{reimann:apal} has a similar development of representations of measures.) If $\mu \in \Meas$, any real $r$ with $\rho(r) = \mu$ is called a \emph{representation} of $\mu$. Note that the representation of a measure is not unique. In fact, a measure may have uncountably many distinct representations.

\begin{proposition} \label{prop:repres-relation}
  Let $r \in \Cant$ be a representation of a measure $\mu \in \Meas$. Then the relations
  \[
     \mu\Cyl{\sigma} < q \quad \text{ and } \quad  \mu\Cyl{\sigma} > q \qquad (\sigma \in \Str, q \in \Rat)
  \] 
  are c.e.\ in $r$.
\end{proposition}

\begin{proof}
  \[
     |\nu_{\gamma(k)}\Cyl{\sigma} - \mu\Cyl{\sigma}| \leq 2^{-k+|\sigma|+2}.
  \]
  This in turn implies 
  \[
     \mu\Cyl{\sigma} \leq \nu_{\gamma(k)}\Cyl{\sigma} + 2^{-k+|\sigma|+2} \leq \mu\Cyl{\sigma} + 2^{-k+|\sigma|+3}.
   \] 
  Hence $\mu\Cyl{\sigma} < q$ if and only if 
  \[
    \exists k \; \nu_{\gamma(k)}\Cyl{\sigma} + 2^{-k+|\sigma|+2} < q.
  \]
  But $\nu_{\gamma(k)}\Cyl{\sigma}$ is a dyadic rational uniformly computable in $r$, and hence $\nu_{\gamma(k)}\Cyl{\sigma} + 2^{-k+|\sigma|+2} < q$ is decidable given $r$ as an oracle. The proof for $\mu\Cyl{\sigma} > q$ is symmetrical.
\end{proof}

Proposition \ref{prop:repres-relation} easily implies the following.

\begin{proposition} \label{prop:computation-measure}
  Let $r \in \Cant$ be a representation of a measure $\mu \in \Meas$. Then $r$ computes a function $g_\mu : \Str \times \Nat
\to \Rat$ such that for all $\sigma \in \Str$, $n \in \Nat$, 
$$
  |g_\mu(\sigma,n) - \mu\Cyl{\sigma}| \leq 2^{-n}.
$$
\end{proposition}

%
%
\section{Randomness and Transformations of Measures} \label{sec:rand-trans}

%
%
\subsection{Random Reals} \label{sub:random-reals}

We define (relative) randomness of reals for arbitrary measures as a straightforward extension of Martin-L\"of's test notion. The basic idea is to require the test to be \emph{enumerable} ($\Sigma^0_1$) in a representation of the measure. 

\begin{definition}
Let $r_\mu$ be a representation of a measure $\mu$, and let $z \in \Cant$.

\begin{enumerate}[(a)]
 	\item An \emph{$(r_\mu,z)$-test} is given by a sequence $(V_n \colon n \in \Nat)$ of uniformly $\Sigma^0_1(r_\mu \oplus z)$-sets $V_n \subseteq \Str$ such that for all $n$, 
  \[
    \sum_{\sigma \in V_n} \mu\Cyl{\sigma} \leq 2^{-n}.
  \]

 	\item A real $x \in \Cant$ \emph{passes} an $(r_\mu,z)$-test $(V_n)$ if $x \not\in \bigcap_n \Cyl{V_n}$. Otherwise we say the test $(V_n)$ \emph{covers} $x$.

 	\item A real $x \in \Cant$ is $(r_\mu,z)$-\emph{random} if it passes all $(r_\mu,z)$-tests.
 
 \end{enumerate} 
\end{definition}

If, in the previous definition, $z = \emptyset$ (where we identify reals with subsets of natural numbers via the characteristic sequence), we simply speak of an $r_\mu$\emph{-test} and of $x$ being \emph{$r_\mu$-random}. 

The previous definition defines randomness with respect to a specific representation. If $x$ is random for one representation, it is not necessarily random for other representations. 
On the other hand, we can ask whether a real exhibits randomness with respect to \emph{some} representation, so the following definition makes sense.
 
\begin{definition} \label{def:mu-rand}
A real $x \in \Cant$ is \emph{$\mu$-random relative to $z \in \Cant$}, or simply \emph{$\mu$-$z$-random}, if there exists a representation $r_\mu$ of $\mu$ so that $x$ is $(r_\mu,z)$-random. 
\end{definition}

One might argue that this definition of randomness is subject to a certain arbitrariness, as it depends on a particular representation of a measure. 
However, it has recently been shown by \citet{day-miller:randomness-non-computable_2011} that the definition of randomness given in Definition \ref{def:mu-rand} is equivalent to the representation-independent approach via \emph{uniform tests} due to \citet{levin:1973, levin:1976, levin:1984} and \citet{gacs:2005}.

Moreover, in this paper we are interested in results of the type ``For which reals does there exist a measure for which $x$ looks non-trivially random?'' -- i.e.\ can $x$ look random at all? If there exists such a (representation of a) measure, there is good reason to say that $x$ has \emph{some} random content. It is not the aim of this paper to provide a most general solution to the problem of defining randomness for arbitrary measures. The goal is to exhibit an interesting connection between the randomness properties of reals (with respect to \emph{some} representation) and its logical complexity.

\medskip
Of course, every real $x$ is trivially $\mu$-random if it is a $\mu$-atom. The question is under what circumstances $x$ is \emph{non-trivially} $\mu$-random, i.e.\ when does there exist a measure $\mu$ so that $x$ is $\mu$-random and $\mu\{x\} = 0$. 

\medskip
A most useful property of the theory of Martin-Löf randomness is the existence of \emph{universal tests}. Universal tests subsume all other tests. Furthermore, they can be defined uniformly with respect to any parameter. More precisely, for any representation $r_\mu$ of a measure $\mu$, there exists a uniformly c.e.\ in $r_\mu$ sequence $(U_n \colon n \in \Nat)$ of sets $U_n \subseteq \Str$ such that, if we set for $z \in \Cant$,
\[
	U^{z}_n = \{ \sigma \colon \Tup{\sigma, \tau} \in U_n, \: \tau \Sle z \},
\]
then $(U^{z}_n)$ is an $(r_\mu,z)$-test and $x \in \Cant$ is $(r_\mu,z)$-random if and only if $x$ passes $(U^{z}_n)$. We call $(U_n)$ a \emph{universal oracle test} for $r_\mu$. For details on the existence of universal tests, see \citep{Day:2011a}. One can also construct universal tests that have a larger number of parameters. In Section \ref{sec:arbi-rand} we will need a test that is universal for two parameters. Such a test is a uniformly c.e.\ in $r_\mu$ sequence $(U^{(2)}_n \colon n \in \Nat)$ so that, if we set for $z_0, z_1 \in \Cant$,
\[
  U^{z_0,z_1}_n = \{ \sigma \colon \Tup{\sigma, \tau_0, \tau_1} \in U^{(2)}_n, \; \tau_0 \Sle z_0, \tau_1 \Sle z_1 \},
\]
then $(U^{z_0,z_1}_n)$ is an $(r_\mu,z_0\oplus z_1)$-test and $x \in \Cant$ is $(r_\mu,z_0\oplus z_1)$-random if and only if $x$ passes $(U^{z_0,z_1}_n)$.
 
\bigskip

\subsection{Computable measures} 
\label{sub:computable_measures}

A measure is \emph{computable} if there exists a computable function $g: \Str \times \Nat
\to \Rat$ such that for all $\sigma \in \Str$, $n \in \Nat$, 
$$
  |g(\sigma,n) - \mu\Cyl{\sigma}| \leq 2^{-n}.
$$
By Proposition \ref{prop:computation-measure}, any measure with a computable representation is computable. The converse holds, too (see \citep{hoyrup-rojas:computabilityprobability_2009}). For a computable measure $\mu$, not being $\mu$-$z$-random is equivalent to the existence of a sequence $(V_n)$ uniformly c.e.\ in $z$ so that for all $n$, $\mu\Cyl{V_n} \leq 2^{-n}$ and so that $(V_n)$ covers $x$. We call the latter a $(\mu,z)$-test. 

Lebesgue measure $\Leb$ is computable, and it is arguably the most prominent measure on $\Cant$. $\Leb$-random reals are the most studied ones by far, and we will, in consistency with the literature, use the name \emph{Martin-L\"of random reals} for them.

Reals that are random with respect to some computable probability
measure have been called \emph{proper} \citep{zvonkin-levin:1970} or
\emph{natural} \citep{muchnik-semenov-uspensky:1998}. 

Obviously no computable real can be non-trivially random with respect to any measure. The following observation yields that the trivially random reals with respect to computable measures are precisely the computable reals. 

%
%
\begin{proposition}[Levin, \citeyear{zvonkin-levin:1970}]\label{entrop:pro_comp-atoms}
  If $\mu$ is a computable measure and $\mu \{x\} > 0$ for some
  $x \in \Cant$, then $x$ is computable.
\end{proposition}
 
\begin{proof}
  Suppose $\mu\{x\} > 2^{-m} > 0$ for some computable $\mu$ and $m \geq 1$.
  Let $g$ be a computation function for $\mu$, i.e.\ $g$ is computable and for all $\sigma$ and $n$, $|g(\sigma,n) - \mu\Cyl{\sigma}| \leq 2^{-n}$. Define a computable tree $T \subseteq \Str$ by letting $\sigma \in T$ if and only if
  $g(\sigma,|\sigma|) \geq 2^{-m} - 2^{-|\sigma|}$. $x$ is an infinite path through $T$. We claim that $x$ is a \emph{fully isolated} path, i.e.\ there exists a string $\sigma$ such that for all $\tau \Sgeq \sigma$, $\tau \in T$ implies $\tau \Sle x$. Clearly any fully isolated path through $T$ is computable. 

  Suppose $x$ were not fully isolated. Then there exist infinitely many $\sigma_n \Sle x$, $\sigma_n \Sle \sigma_{n+1}$, such that $\sigma_n^{\vee} \in T$, where $\sigma_n^{\vee}$ is obtained from $\sigma_n$ by switching the last bit. It follows that the $\sigma_n^{\vee}$ are pairwise incompatible. Since $\sigma_n^{\vee}\in T$, we have that for sufficiently large $n$, 
  \[
  	\mu\Cyl{\sigma_n^{\vee}} \geq \frac{1}{2^{m+1}},
  \]	
  which is impossible since the $\sigma_n^{\vee}$ are pairwise incompatible and $\mu$ is a probability measure.
\end{proof}

As regards non-computable reals, Levin proved that, from a computability
theoretic point of view, randomness with respect to a computable
probability measure is computationally as powerful as Martin-L\"of
randomness. This was
independently shown by \citet{kautz:1991}.

%
%
\begin{theorem}[Levin, \citeyear{zvonkin-levin:1970}; Kautz
  \citeyear{kautz:1991}] \label{rand:thm_propalgequiv} 
  A non-computable real which is random with respect to some computable measure is Turing  equivalent to a Martin-L\"of random real.
\end{theorem}

The proof of Theorem \ref{rand:thm_propalgequiv} uses the fact that reductions
induce continuous (partial) mappings from $\Cant$ to $\Cant$. Such
mappings transform measures. 


%
%
\subsection{Transformation of Measures}

Let $\mu$ be a Borel measure on $\Cant$, and let $f: \Cant \to \Cant$
be a $\mu$-measurable function, that is, for all measurable $A \subseteq \Cant$, $f^{-1}(A)$ is $\mu$-measurable, too. Such $f$ induces a new measure $\mu_f$, often referred to as
the \emph{image measure} or \emph{push-forward} of $\mu$, on $\Cant$ by letting
\[
        \mu_f(\Cl{A}) = \mu(f^{-1}(\Cl{A})).
\]

Any non-atomic measure 
can be transformed into Lebesgue measure $\Leb$ this way. Recall that a function $f:\Cant \to \Cant$ is a \emph{Borel automorphism} if it is bijective and for any $A \subseteq \Cant$, $A$ is a Borel set if and only if $f(A)$ is. 

%
%
\begin{theorem}[see \citep{halmos:1950}] \label{rand:thm_hommeas}
	Let $\mu$ be a non-atomic probability measure on $\Cant$.
  	\begin{enumerate}[(1)]
	\item There exists a continuous mapping $f: \Cant \to \Cant$ such that $\mu_f = \Leb$.
	\item There exists a Borel automorphism $h$ of $\Cant$ such that $\mu_h = \Leb$.
	\end{enumerate}
\end{theorem}
                                        
The simple idea to prove (1) is to identify Cantor space with the unit interval $[0,1]$ and define $f(x) = \mu([0,x])$, that is, to let $f$ equal the distribution function of $\mu$.

To prove Theorem \ref{rand:thm_propalgequiv}, Levin showed that this idea works in the effective case, too. Furthermore, one can even deal with the presence of atoms, as long as the measure is computable.
Roughly speaking, a
computable transformation transforms a computable measure into another
computable measure. At the same time, a computable transformation
will preserve randomness in the sense that a real random with
respect to $\mu$, when transformed by a computable mapping,
is random with respect to the image measure of $\mu$. 

In the following, we will use a (relativized) transformation of Lebesgue measure to render a given real random.

%
%
\section{Randomness with Respect to Arbitrary Measures}\label{sec:arbi-rand}

The Levin-Kautz result implies that the set of reals that are non-trivially random with respect to a computable measure are contained in the set of Martin-L\"of random Turing degrees. But what about the reals non-trivially random with respect to an arbitrary measure? In this section we will show that these coincide with the non-computable reals. Every non-computable real is non-trivially random with respect to some measure.

The proof of this result uses two important results from computability theory and algorithmic randomness, the Ku{\v{c}}era-G{\'a}cs Theorem and the Posner-Robinson Theorem.

We will need the Ku{\v{c}}era-G{\'a}cs Theorem in the following, relativized form.
%
%
\begin{theorem}[\citet{kucera:1985, gacs:1986}]\label{arbi:thm_rel-kucera}
  Let  $x, z \in \Cant$. There exists a real $y$ that is Martin-Löf random relative to $z$ such that
  \[
    x \leq_{\WTT(z)} y \leq_{\WTT(z)} x \oplus z'.
  \]
\end{theorem}

%
%
\begin{theorem}[\citet{posner-robinson:1981}]
	\label{entrop:thm_posner-robinson}
If $x \in \Cant$ is non-computable, then there is a $z \in
  \Cant$ such that $x \oplus z \geq_{\T} z'$.
\end{theorem}

\begin{corollary}\label{cor:kg-pr}  
  For every non-computable real $x$, there exist reals $y,z \in \Cant$ so that $y$ is Martin-Löf random relative to $z$ and 
  \[
    x \leq_{\WTT(z)} y \quad \text{ and }  \quad y \leq_{\T(z)} x. 
  \] 
\end{corollary}

The (relative) Turing equivalence to a random real allows for
transforming Lebesgue measure $\Leb$ in a sufficiently controlled
manner. More precisely, we can obtain a $\Pi^0_1$ class $M$ of
representations of measures. Each measure represented in this class is a good candidate for a
measure that renders $x$ random. We will use a compactness
argument to show that at least one member $r_\mu$ of $M$ has the
property that the Martin-L\"of random real $y$ is still
$\Leb$-random \emph{relative to} $r_\mu$ (here $r_\mu$ is viewed as a real, not a measure). Then, $x$ has to be $r_\mu$-random, since otherwise an
$r_\mu$-test could be effectively transformed into a
Martin-L\"of $\Leb$-test relative to $r_\mu$ which $y$ would fail.

\bigskip
We now state and prove the main result of this section.

%
%
\begin{theorem}\label{thm:nonrec-impl-rand}
  For any real $x \in \Cant$, the following are
  equivalent: 
  \begin{enumerate}[(i)]
  \item There exists a probability measure $\mu$ such that $x$ is
    not a $\mu$-atom and $x$ is $\mu$-random.
  \item $x$ is not computable.
  \end{enumerate}
\end{theorem}

\begin{proof}
(i) $\Rightarrow$ (ii): If $x$ is computable and $\mu$ is a measure with $\mu\{x\} = 0$, then
we can construct a $\mu$-test that covers $x$ by
using $x$ and any representation of $\mu$ to search for an initial segments of
$x$ whose measure is sufficiently small. More formally, given $n$, compute, using any representation $r_\mu$ of $\mu$ as an oracle, a length $l_n$ for which $\mu\Cyl{x\Rest{l_n}} < 2^{-n}$. Define a $\mu$-test $(V_n)$ by letting $V_n = \{x\Rest{l_n}\}$.
 
(ii) $\Rightarrow$ (i): 
Let $x$ be a non-computable real. Using Corollary \ref{cor:kg-pr}, we obtain 
a real $y$ which is Martin-L\"of random relative to some $z \in \Cant$ and which is
$\T(z)$-equivalent to $x$. 

There are Turing $z$-functionals
$\Phi$ and $\Psi$ such that 
\begin{displaymath}
  \Phi(y) = x \quad \text{ and } \quad \Psi(x) = y.
\end{displaymath}

We will use the functionals $\Phi$ and $\Psi$ to define a set of measures $M$. 
If $\Phi$ were total and invertible, there
would be no problem to define the desired measure, as one could simply
`push forward' Lebesgue measure using $\Phi$. In our case we have to
use $\Phi$ and $\Psi$ to control the measure. We are guaranteed that
this will work \emph{locally}, since $\Phi$ and $\Psi$ are mutual inverses when restricted to
$x$ and $y$. Therefore, given a string $\sigma$ (a possible
initial segment of $x$) we will single out strings which appear
to be candidates for initial segments of an inverse real.

For any $\sigma$, let $\Op{Pre}^*(\sigma) \subseteq \Str$ be defined as
\begin{equation*}
 \Op{Pre}^*(\sigma) = \{ \tau \in \Str: \: \Phi(\tau) \Sgeq \sigma \: \And \:
\Psi_{|\sigma|}(\sigma) \Sleq \tau \}.
\end{equation*}
We will need only the elements of $\Op{Pre}^*$ that are minimal with respect to the prefix relation. Let
\begin{equation*}
 \Op{Pre}(\sigma) = \{ \tau \in \Op{Pre}^*(\sigma)\colon \forall \tau' \in \Op{Pre}^*(\sigma) \: (\tau, \tau' \text{ compatible }  \to \tau \Sleq \tau') \}
\end{equation*}
Note that $\Op{Pre}(\sigma)$ is uniformly c.e.\ in $z$, since we can approximate $\Phi(\tau)$ by longer and longer prefixes (the strings $\Phi_s(\tau)$), and we assume the reductions $\Phi, \Psi$ to be use monotone.

\medskip
To define a measure $\mu$ with respect to which $x$ is non-trivially
random, we satisfy two requirements:
\begin{enumerate}[(1)]
\item The measure $\mu$ \emph{dominates} the partial push-forward of Lebesgue measure induced by
   $\Phi$. This will help ensure that any Martin-L\"of random real is
  mapped by $\Phi$ to a $\mu$-random real.
\item The measure $\mu$ \emph{must not be atomic} on $x$. 
\end{enumerate}
To meet these requirements, we restrict the values of $\mu$ in the
following way:
\begin{equation}\label{entrop:equ_meas-cond}
  \Leb\Cyl{\Op{Pre}(\sigma)} \leq \mu\Cyl{\sigma} \leq \Leb\Cyl{\Psi_{|\sigma|}(\sigma)}. 
\end{equation}
The first inequality ensures that (1) is met, whereas the second
guarantees that $\mu$ is non-atomic on the domain of $\Psi$ (since if $\Psi(z)$ is defined, then $\lim_s |\Psi_s (z\Rest{s})| = \infty$ and thus $\Leb\Cyl{\Psi_s (z\Rest{s}))} \to 0$). If $\Psi(z)$ is undefined, then $\Psi_s (z\Rest{s})$ is constant from some point on and hence imposes a constant positive upper bound on all $\mu\Cyl{z\Rest{s}}$ from that point on.

Let $M \subseteq \Cant$ be the set of all representations of measures that satisfy \eqref{entrop:equ_meas-cond}.
We show that $M$ is non-empty and $\Pi^0_1(z)$.

\begin{claim*}
The set $M$ is not empty.
\end{claim*}

\begin{proof}
We exhibit a measure $\mu$ that respects all upper and lower bounds given by (\ref{entrop:equ_meas-cond}). Since every measure has a representation, this implies that $M$ is non-empty. We construct $\mu$ inductively on the basic open cylinders. Put $\mu\Cyl{\Estr} = 1$. Suppose $\mu\Cyl{\sigma}$ is given such that 
\begin{equation*}
   \Leb\Cyl{\Op{Pre}(\sigma)} \leq \mu\Cyl{\sigma} \leq \Leb\Cyl{\Psi_{|\sigma|}(\sigma)}
\end{equation*}
It follows from the definition of $\Op{Pre}$ that 
\[
	\Op{Pre}(\sigma \Conc 0), \Op{Pre}(\sigma\Conc 1) \subseteq \Op{Pre}(\sigma) \quad \text{ and } \quad \Op{Pre}(\sigma\Conc 0) \cap \Op{Pre}(\sigma\Conc 1) = \emptyset.
\] 
Hence, since $\Leb$ is a measure,
\begin{equation*}
  \Leb\Cyl{\Op{Pre}(\sigma\Conc 0)} + \Leb\Cyl{\Op{Pre}(\sigma\Conc 1)} \leq \Leb\Cyl{\Op{Pre}(\sigma)}. 
\end{equation*}
Furthermore, by the properties of the approximation $\Psi_s$ stated in Section \ref{sub:turing_functionals}, we have 
\begin{equation*}
  \Leb\Cyl{\Psi_{|\sigma\Conc 0|}(\sigma\Conc 0)} + \Leb\Cyl{\Psi_{|\sigma\Conc 1|}(\sigma\Conc 1)} \geq 2^{-|\Psi_{|\sigma|}(\sigma)|-1} + 2^{-|\Psi_{|\sigma|}(\sigma)|-1} = \Leb\Cyl{\Psi_{|\sigma|}(\sigma)}
\end{equation*}
Thus,
\begin{equation*}
   \Leb\Cyl{\Op{Pre}(\sigma\Conc 0)} + \Leb\Cyl{\Op{Pre}(\sigma\Conc 1)} \leq \mu\Cyl{\sigma} \leq \Leb\Cyl{\Psi_{|\sigma\Conc 0|}(\sigma\Conc 0)} + \Leb\Cyl{\Psi_{|\sigma\Conc 1|}(\sigma\Conc 1)}.
\end{equation*}
Since the mapping $\theta: [0,1] \times [0,1] \to \Real$ given by 
\begin{multline*}
  \theta(s,t) =  \Leb\Cyl{\Op{Pre}(\sigma\Conc 0)} + s \bigl ( \Leb\Cyl{\Psi_{|\sigma\Conc 0|}(\sigma\Conc 0)} -  \Leb\Cyl{\Op{Pre}(\sigma\Conc 0)} \bigr ) \; + \; \\ \Leb\Cyl{\Op{Pre}(\sigma\Conc 1)} + t \bigl ( \Leb\Cyl{\Psi_{|\sigma\Conc 1|}(\sigma\Conc 1)} -  \Leb\Cyl{\Op{Pre}(\sigma\Conc 1)} \bigr )
\end{multline*}
is continuous, it follows from the intermediate value theorem that there exist $s_0,t_0$ such that $\theta(s_0,t_0) = \mu\Cyl{\sigma}$. Put 
\begin{align*}
  \mu\Cyl{\sigma\Conc 0} & = \Leb\Cyl{\Op{Pre}(\sigma\Conc 0)} + s_0 \bigl ( \Leb\Cyl{\Psi_{|\sigma\Conc 0|}(\sigma\Conc 0)} -  \Leb\Cyl{\Op{Pre}(\sigma\Conc 0)} \bigr) \\
  \mu\Cyl{\sigma\Conc 1} & = \Leb\Cyl{\Op{Pre}(\sigma\Conc 1)} + t_0 \bigl ( \Leb\Cyl{\Psi_{|\sigma\Conc 1|}(\sigma\Conc 1)} -  \Leb\Cyl{\Op{Pre}(\sigma\Conc 1)} \bigr ) 
\end{align*}
\end{proof}

\begin{claim*}
	The set $M$ is $\Pi^0_1(z)$.
\end{claim*}

\begin{proof}
  A representation $r$ is in $M$ if and only if 
 \[
 	\forall \sigma \: \Leb\Cyl{\Op{Pre}(\sigma)} \leq \rho(r)\Cyl{\sigma} \leq \Leb\Cyl{\Psi_{|\sigma|}(\sigma)}. 
\]
It suffices to show that the relation $\Leb\Cyl{\Op{Pre}(\sigma)} \leq \rho(r)\Cyl{\sigma} \leq \Leb\Cyl{\Psi_{|\sigma|}(\sigma)}$ is uniformly $\Pi^0_1(z)$. 

   The set $\Op{Pre}(\sigma)$ is c.e.\ in $z$ (uniformly in $\sigma$), and thus the measure $\Leb\Cyl{\Op{Pre}(\sigma)}$ is left-enumerable in $z$. There exists a strictly increasing, computable sequence of dyadic rationals $(q_n)$ so that $q_n \to \Leb\Cyl{\Op{Pre}(\sigma)}$. We have that
  \[
     \rho(r)\Cyl{\sigma} < \Leb\Cyl{\Op{Pre}(\sigma)} \quad \Leftrightarrow \quad \exists n \: \rho(r)\Cyl{\sigma} < q_n.
  \]
  By Proposition \ref{prop:repres-relation}, $\rho(r)\Cyl{\sigma} < q_n$ is $\Sigma^0_1(z)$. Hence $\rho(r)\Cyl{\sigma} < \Op{Pre}(\sigma)$ is $\Sigma^0_1(z)$, too. Proposition \ref{prop:repres-relation} also yields that the relation $\rho(r)\Cyl{\sigma} > \Leb\Cyl{\Psi_{|\sigma|}(\sigma)}$ is $\Sigma^0_1(z)$. Hence the relation $\Leb\Cyl{\Op{Pre}(\sigma)} \leq \rho(r)\Cyl{\sigma} \leq \Leb\Cyl{\Psi_{|\sigma|}(\sigma)}$ is $\Pi^0_1(z)$, as desired, and the argument above is uniform in $\sigma$.
\end{proof}

\medskip
Let $r_\mu \in M$. It follows from the definiton of $M$ that $x$ is not an atom of $\mu$. Suppose $(V_n)$ is an $(r_\mu,z)$-test that covers $x$. For each $n$, let
\[
  W_n =  \{ \Op{Pre}(\sigma) \colon \sigma \in V_n \}.
\]
Then $W_n$ is uniformly $\Sigma^0_1(r_\mu\oplus z)$ since $\Op{Pre}(\sigma)$ is uniformly c.e.\ in $z$. Furthermore, 
\[
  \lambda \Cyl{W_n} \leq \sum_{\sigma \in V_n} \lambda\Cyl{\Op{Pre}(\sigma)} \leq \sum_{\sigma \in V_n} \mu\Cyl{\sigma} \leq 2^{-n}.
\]
Hence $(W_n)$ is a $(\lambda,r_\mu\oplus z)$-test. Moreover, since $y \in \Cyl{\Op{Pre}(x\Rest{k})}$ for all $k$ and $(V_n)$ covers $x$, $(W_n)$ covers $y$. 

At this point, we are not yet quite able to derive a contradiction, since we only have that $y$ is Martin-Löf random relative to $z$, and not necessarily relative to $r_\mu \oplus z$. However, we can use the following basis theorem for $\Pi^0_1$-classes to infer the existence of an element $r \in M$ so that $y$ is $(\lambda,r \oplus z)$-random. Lemma \ref{entrop:thm_rand-conserv-basis} was independently obtained by \citet{downey-hirschfeldt-miller-nies:2005}.

%
%
\begin{lemma}\label{entrop:thm_rand-conserv-basis}
  Let $z \in \Cant$, and let $T \subseteq \Str$ be a
  tree computable in $z$ such that $T$ has an infinite path. Then, for every real
  $y$ that is Martin-L\"of random relative to $z$, there is
  an infinite path $v$ through $T$ such that $y$ is Martin-L\"of
  random relative to $z \oplus v$.
\end{lemma} 

\begin{proof}
  We use the universal oracle test $(U^{(2)}_n)$ introduced in Section \ref{sec:rand-trans}. 
  Given $z \in \Cant$ and $\tau \in \Str$, let
  \[
    U^{z,\tau}_n = \{ \sigma \colon \Tup{\sigma, \tau_0, \tau_1} \in U^{(2)}_n, \;  \tau_0 \Sle z,  \tau_1 \Sleq \tau \},
  \]
  where $(U^{(2)}_n)$ is a universal $\lambda$-test for two parameters, as introduced in Section \ref{sub:random-reals}. Note that the sequence $(U^{z,\sigma}_n)$ is uniformly c.e.\ in $z$ and forms a $(\lambda,z)$-test.

  We enumerate a $(\lambda,z)$-test $(V_n)$ as follows: enumerate a string
  $\sigma$ into $V_n$ if $\Cyl{\sigma}$ is contained in
  $\ACyl{U^{z,\tau}_n}$ for all $\tau \in T$ with $|\tau| =
  |\sigma|$ (note that there are only finitely many such $\tau$).
  
  If $y$ is Martin-L\"of random relative to $z$, there has to
  be some $n$ such that $y \not\in \Cyl{V_n}$. 
  Let 
  \[
    T_0 = \{ \tau \in T \colon \Cyl{y\Rest{|\tau|}} \text{ is not contained in } \Cyl{U^{z,\tau}_n} \}.
  \]
  $T_0$ is clearly closed under initial segments, hence it is a subtree of $T$. For every
  $m$, $y\Rest{m}$ is not enumerated in $V_n$, which implies that $T_0$ has nodes of every length, in particular it is infinite. 
  Applying  K\"onig's Lemma yields an infinite path $v$ through $T_0$.

  Now, $y$ is Martin-Löf random relative to $z \oplus v$. For suppose not, then, since $(U^{(2)}_n)$ is a universal oracle test for $\lambda$, $y \in \bigcap_k \Cyl{U^{z,v}_k}$. In particular, $y \in \Cyl{U^{z,v}_n}$. By the use principle, there is an initial segment $\tau \Sle v$ such that $y\in \Cyl{U^{z,\tau}_n}$, contradicting the fact that $\tau \in T_0$.  
\end{proof}

Using Lemma \ref{entrop:thm_rand-conserv-basis}, we obtain a representation $r_\mu \in M$ so that $x$ is $(r_\mu,z)$-random.
To complete the proof of Theorem \ref{thm:nonrec-impl-rand}, note that any $(r_\mu,z)$-random real is also $r_\mu$-random, and hence $\mu$-random.
\end{proof}

How hard is it to find a measure that makes a given real $x$ random? The tree determining $M$ is computable in the parameter $z$, stemming from the Posner-Robinson Theorem. $z$ can be found recursively in $x \oplus \emptyset'$. It takes another jump to find a path through the tree which conserves randomness for some $y$ Turing $z$-equivalent to $x$. Hence, there is a representation $r_\mu \leq_{\T} x''$ of a measure $\mu$ so that $x$ is $r_\mu$-random.

%
%
\section{Randomness with Respect to Continuous
Measures}\label{entrop:sec_rand-with-resp}

In this section we investigate what happens if one replaces the property ``\emph{random with respect to an arbitrary probability measure}'' (which turned out to hold for any non-computable real) with \emph{being random for a continuous probability measure}.

First, we give an explicit construction of a non-computable real which does not have the latter property.

%
%
\subsection{A real not continuously random}

\begin{theorem}\label{thm:exists-ncr}
There exists a non-computable real which is not random with respect to any continuous measure.
\end{theorem}

\begin{proof}
Consider the halting problem $\Halt$. Denote by $\Halt_t$ the approximation to $\Halt$ enumerated after $t$ steps. We define a set of markers $\gamma_t(n)$ that capture when the first $n$ bits of $\Halt_t$ have settled. For each $t$, let $\gamma_t(0) = 0$ and define
\begin{equation*}
  \gamma_t(n+1) := \max \{ \min \{s \leq t : \: \Halt_s \Rest{n+1} = \Halt_t \Rest{n+1} \}, \gamma_t(n)+1 \}.
\end{equation*}
Each $\gamma_t(n)$ (as a function of $t$) will be constant from some point on as computations of $\Halt$ settle. We denote this value by $\gamma(n)$. Let $y_t$ the real given by the characteristic sequence of $\{\gamma_t(n) \colon n \geq 0\}$, and let $y$ be the real given by the characteristic sequence of $\{\gamma(n) \colon n \geq 0\}$.
We claim that $y$ is not random with respect to any continuous measure.

Let $\mu \in \Meas$ be continuous, and let $r_\mu$ be any representation of $\mu$. Since $\Cant$ is compact and $\mu$ is continuous, for every rational $\varepsilon > 0$ there exists a number $l(\varepsilon)$ such that 
\begin{equation*}
  \forall \sigma \in \Bit^{l(\varepsilon)} \; \mu \Cyl{\sigma} < \varepsilon.
\end{equation*}
By Proposition \ref{prop:repres-relation}, the relation $\mu\Cyl{\sigma} < \varepsilon$ is $\Sigma^0_1(r_\mu)$. Therefore, $r_\mu$ can compute such a function $l$.
We use this to uniformly enumerate the $n$th level of an $r_\mu$-test $(V_n)_{n \in \Nat}$ that covers $y$.

First, compute $n_0 = l(2^{-n-1})$ and $n_1 = l(2^{-n-1}/n_0)$.  
Enumerate $y_{n_1} \Rest{n_0}$ into $V_n$. Furthermore, for all $k$ such that $\gamma_{n_1}(k) < n_0$, enumerate the string 
\[
  \bigl (y_{n_1} \Rest{\gamma_{n_1}(k)} \bigr)  \Conc 0^{n_1 - \gamma_{n_1}(k)}
\] 
into $V_n$. 

Note that $y_{t}\Rest{n_0}$ can change at most $n_0$ many times. Observe further that, if the approximation to $\Halt$ changes at a position $m$ at time $t$, it will move the marker $\gamma_t(m)$ and with it all the markers $\gamma_t(k)$, $k > m$ to a position $\geq t$. Hence, for the maximum $k \leq n_0$ so that $\gamma_{n_1}(k) = \gamma(k)$, 
\[
  y_{n_1}\Rest{\gamma_{n_1}(k)} \; = \; y\Rest{\gamma_{n_1}(k)} \; = \; y\Rest{\gamma(k)},
\]
and between $\gamma_{n_1}(k)$ and $n_1$, the characteristic sequence of $y$ has only $0$s. Therefore, $y$ is covered by $V_n$.

Finally, note that the measure of $V_n$ is at most $2^{-n}$, since
\begin{equation*}
  \begin{split}
  	\sum_{v \in V_n} \mu \Cyl{v} & = \mu \Cyl{y_{n_1} \Rest{n_0}} + \sum_{\{k\colon \gamma_{n_1}(k) < n_0\}} \mu \Cyl{(y_{n_1} \Rest{\gamma_{n_1}(k)})  \Conc 0^{n_1 - \gamma_{n_1}(k)}} \\
  	& \leq 2^{-n-1} + n_0 \frac{2^{-n-1}}{n_0} = 2^{-n}.
  \end{split}
\end{equation*}
\end{proof}

%
%
\subsection{Classifying the not continuously random reals}

Theorem \ref{thm:exists-ncr} suggests the following question: Is it possible to classify the reals which are not random with respect to a continuous measure? Denote by $\NCR$ the set of all such reals. Can we obtain bounds on the complexity of $\NCR$? 

An observation by \citet{kjoshanssen-montalban:2005} shows that $\NCR$ is cofinal the hyperarithmetical Turing degrees, i.e.\ for any hyperarithmetical $x$ there exists a hyperarithmetical $y \in \NCR$ so that $x \leq_{\T} y$.

It follows directly from the countable additivity of measures that, for every continuous measure $\mu$, any countable subset of $\Cant$ has $\mu$-measure zero. 

For countable $\Pi^0_1$ classes, we can strengthen this to effective $\mu$-measure zero, and hence no countable $\Pi^0_1$ class contains a real in $\NCR$.

%
%
\begin{theorem}[\citep{kjoshanssen-montalban:2005}]
If $A \subseteq \Cant$ is countable and $\Pi^0_1$, then no member of $A$ can be random with respect to a continuous measure.
\end{theorem}

\begin{proof}
	Let $A = [T]$ for some computable tree $T$, and suppose $\mu$ is a continuous measure. Let $r_\mu$ be any representation of $\mu$. Let $T^{=n}$ denote the strings in $T$ of length $n$. It holds that
  \[
    \mu\Cyl{T^{=n}} \to \mu [T] = 0 \qquad (n \to \infty).
  \]
  Using Proposition \ref{prop:repres-relation}, we can browse the tree $T$ level by level till we see that the measure of $\Cyl{T^{=n}}$ falls below $2^{-n}$. When this happens, we enumerate all strings in $T^{=n}$ into the $n$-th level of an $r_\mu$-test.
\end{proof}

\citet{kreisel:1959} showed that every member of a countable $\Pi^0_1$ class is hyperarithmetic, i.e.\ contained in $\Delta^1_1$. Furthermore, he showed that members of countable $\Pi^0_1$ classes (also called \emph{ranked points}) can be found cofinally the hyperarithmetical Turing degrees. Later, \citet{cenzer-etal:1986} showed that ranked points appear at each hyperarithmetical level of the Turing jump, i.e.\ can be found in Turing degrees obtained by iterating the Turing jump along a computable ordinal. 

%
%
\begin{corollary}\label{cont:cor_ncr-hyp}
	For every computable ordinal $\beta$ there exists an $x \in \NCR$ such that $x \equiv_{\T} \emptyset^{(\beta)}$.
\end{corollary}

We would like to obtain an upper bound on the complexity of reals in $\NCR$; in particular, we would like to know whether $\NCR$ is countable. We start with the following simple observation.

%
%
\begin{proposition}
The set $\NCR$ of all reals not random with respect to any continuous measure is $\Pi^1_1$.
\end{proposition}  

\begin{proof}
	We have
	\begin{multline*}
		x \in \NCR \: \Leftrightarrow \: (\forall r)[ r \text{ represents a measure $\mu$ and $\mu$ is continuous} \\ \: \rightarrow \: \text{some $r$-test covers $x$}].
	\end{multline*}
   The property `$r$ represents a measure' is obviously arithmetic, and due to compactness of $\Cant$, a measure is $\mu$ is continuous if and only if 
\[
	(\forall n)(\exists l)(\forall \sigma)[|\sigma| = l \: \rightarrow \mu\Cyl{\sigma} \leq 2^{-n}].
\]
Hence `$\mu$ is continuous' is arithmetic in $r$. Furthermore `some test covers $x$' can be expressed as
\[
	(\exists e)(\forall n) \; \biggl [ \: \bigl [ (\forall s) \sum_{\sigma \in W_{\{e\}^r_s(n)}}  \rho(r)\Cyl{\sigma} \leq 2^{-n} \bigr ] \; \wedge \;   (\exists \sigma) \bigl [\sigma \in W_{\{e\}^r_s(n)} \: \wedge \: \sigma \Sle x \bigr ] \: \biggr ],
\]                                            
which is arithmetic in $x$ and $r$.
\end{proof}

Furthermore, it is not hard to see that $\NCR$ does not have a perfect subset.

%
%
\begin{proposition}
$\NCR$ does not have a perfect subset.
\end{proposition}

\begin{proof}
Assume $\Cl{X} \subseteq \NCR$ is a perfect subset represented by a perfect tree $T$ with $[T] = \Cl{X}$. We devise a measure $\mu$ by setting $\mu\Cyl{\Estr} = 1$, and define inductively
\begin{equation*}
	\mu \Cyl{\sigma\Conc i} = \begin{cases}
		\mu \Cyl{\sigma} & \text{if } \sigma\Conc (1-i) \not\in T, \\
		\tfrac{1}{2}\mu \Cyl{\sigma} & \text{otherwise}.
	\end{cases}
\end{equation*}
i.e.\ we distribute the measure uniformly over the infinite paths through $T$.

Obviously, $\mu$ is continuous, and since $\mu(X) = 1$, $X$ must contain a $\mu$-random real. (The set of $\mu$-random reals is always a set of $\mu$-measure $1$.)
\end{proof}

The \emph{Perfect Subset Property} refers to the principle that every set in a pointclass is either countable or contains a perfect subset. The Perfect Subset Property for $\pmb{\Pi}^1_1$ is not provable in $\mathsf{ZFC}$. G\"odel showed that if $V=L$, then there exists an uncountable $\Pi^1_1$ set without a perfect subset. 
\citet{mansfield:1970} and \citet{solovay:1969} showed that any $\Sigma^1_2$ set without a perfect subset is contained in the constructible universe $L$. 

%
%
\begin{corollary}
	$\NCR$ is contained in G\"odel's constructible universe $L$.
\end{corollary}

The upper bound $L$ appears indeed very crude, and an analysis of the proof technique of Theorem \ref{thm:nonrec-impl-rand} together with a recent result by \citet{woodin:sub} will yield that $\NCR$ is countable. Even more, Corollary \ref{cont:cor_ncr-hyp} is in certain sense optimal: Every real outside $\Delta^1_1$ is random with respect to some continuous measure.

\medskip
In the proof of Theorem \ref{thm:nonrec-impl-rand}, the decisive property which ensured the non-trivial $\mu$-randomness of $x$ was \eqref{entrop:equ_meas-cond}:
\begin{equation*}
  \Leb\Cyl{\Op{Pre}(\sigma)} \leq \mu\Cyl{\sigma} \leq \Leb\Cyl{\Psi_{|\sigma|}(\sigma)}. 
\end{equation*}
Here, the second inequality guarantees that $x$ is not a $\mu$-atom. Although we know $\mu\Cyl{\sigma}$ will converge to $0$ as we consider longer and longer initial segments $\sigma \Sle x$, this may not to be the case for reals other than $x$, as the reduction $\Psi$ from $x$ to $y$ is a Turing reduction. 

If, however, $\Psi$ is a $\WTT$-reduction, we can modify the construction in the proof of Theorem \ref{thm:nonrec-impl-rand} to obtain a continuous measure with respect to which $x$ is random.

\begin{theorem} \label{thm:wtt-cont-rand}
  Let $x \in \Cant$. Suppose there exist reals $y,z \in \Cant$ so that $y$ is Martin-Löf random relative to $z$ and
  \[
    x \equiv_{\WTT(z)} y,
  \]
  then $x$ is random with respect to a continuous measure.
\end{theorem}

\begin{proof}
  The proof is similar to the proof of Theorem \ref{thm:nonrec-impl-rand}, with one important modification. Suppose $\Phi$ and $\Psi$ are $\WTT$ $z$-functionals such that 
  \begin{displaymath}
    \Phi(y) = x \quad \text{ and } \quad \Psi(x) = y.
  \end{displaymath}
  (In fact, for the proof to work it suffices that $\Phi$ is a Turing $z$-functional.) Let $g: \Nat \to \Nat$ be a computable bound on the use of $\Psi$. Since both $x,y$ are non-computable, we have that $g$ is unbounded. We may assume that $g$ is strictly decreasing. Define a computable function $h$ as
  \[
     h(k) = \max \{m \colon g(m) \leq k \}.
  \]
  If, for some $v \in \Cant$, $\Psi(v)$ is a real, then $\Psi(v \Rest{k})$ is a string of length at least $h(k)$ (for all $k$). Since $g$ is strictly increasing, we have that $h(k+1) \leq h(k) + 1$.

  Define the set $\Op{Pre}^*(\sigma)$ now as
  \[
   \Op{Pre}^*(\sigma) = \{ \tau \in \Str: \: \Phi(\tau) \Sgeq \sigma \: \And \: \exists s \: \bigl ( |\Psi_s(\sigma)| \geq h(|\sigma|) \; \& \; \Psi_s(\sigma) \Sleq \tau \bigr ).
  \]
  As before, let $\Op{Pre}$ be the set of minimal elements of $\Op{Pre}^*$ with respect to the prefix relation. Note that $\Op{Pre}$ is c.e.\ in $z$, as before.

  Condition \eqref{entrop:equ_meas-cond} is replaced by 
  \begin{equation}\label{equ:cont-meas-cond}
    \Leb\Cyl{\Op{Pre}(\sigma)} \leq \mu\Cyl{\sigma} \leq 2^{-h(|\sigma|)}, 
  \end{equation}
  and $M$ is the set of all reals $r$ so that $\rho(r)$ satisfies \eqref{equ:cont-meas-cond} for all $\sigma \in \Str$.

  As $h(k) \to \infty$ as $k \to \infty$, it follows that every measure represented in $M$ is continuous. The condition $h(k+1) \leq h(k) + 1$ ensures that the proof showing $M$ is non-empty still goes through. A similar argument as in the proof of Theorem \ref{thm:nonrec-impl-rand} shows that $M$ is $\Pi^0_1(z)$.

  Finally, since $\Psi(x) = y$, $y \in \Cyl{\Op{Pre}(x\Rest{n})}$ for all $n$. This in turn implies that the final part of the argument, transforming a possible test for $x$ into a test for $y$ (with respect to the accordant measures), goes through as well.
\end{proof}

Woodin showed that outside the hyperarithmetical sets, the Posner-Robinson Theorem holds with truth-table equivalence.

%
%
\begin{theorem}[\citet{woodin:sub}] \label{thm:woddin-pr}
If $x \in \Cant$ is not hyperarithmetic, then there is a $z \in
  \Cant$ such that $x \equiv_{\TT(z)} z'$.
\end{theorem}

Combining Theorems \ref{arbi:thm_rel-kucera}, \ref{thm:woddin-pr}, and \ref{thm:wtt-cont-rand} now yields the desired upper bound for $\NCR$.

%
%
\begin{theorem} \label{cont:thm_ncr-in-hyp}
If a real $x$ is not hyperarithmetic, then there exists a continuous measure $\mu$ such that $x$ is $\mu$-random.
\end{theorem}

Theorem \ref{cont:thm_ncr-in-hyp} yields an interesting measure-theoretic characterization of $\Delta^1_1$. The result can also be obtained via a game-theoretic argument using Borel determinacy, along with a generalization of the Posner-Robinson Theorem via Kumabe-Slaman forcing. This is part of a more general argument which shows that for all $n$, the set $\NCR_n$ of reals which are not $n$-random for some continuous measure is countable. Here $n$-random means that a test has access to the $(n-1)$st jump of a representation of the measure. The countability result for $\NCR_n$ has an interesting metamathematical twist. 
This work will be presented in a separate paper \citep{reimann-slaman:ip}.

\section{Acknowledgments} 
\label{sec:acknowledgments}

We thank an anonymous referee for many insightful and helpful comments and suggestions.


\bibliographystyle{abbrvnat}

\end{document}